\documentclass[a4paper,11pt]{article}
\usepackage{amsmath,amsthm,amssymb,amscd,epsfig,esint}
\usepackage{verbatim}
\usepackage[english]{babel}
\newtheorem{thm}{Theorem}
\newtheorem{lema}[thm]{Lemma}

\newtheorem{coro}[thm]{Corollary}

\newtheorem{quest}[thm]{Question}

\newcommand{\deltabarra}{\bar{\partial}\,}
\newcommand{\ind}{\text{Index}\,}
\newcommand{\supp}{\text{supp}\,}

\def\R{\mathbb R}
\def\C{\mathbb C}

\def\D{\mathbb D}

\def\B{\mathcal S}

\def\Im{\operatorname{Im}}

\def\B{\mathcal B}

\def\div{\operatorname{div}}

\def\Id{\mathbf{Id}}

\title{Weighted estimates for Beltrami equations}

\author{Albert Clop \and Victor Cruz}

\date{}

\begin{document}

\maketitle

\abstract{We obtain a priori estimates in $L^p(\omega)$ for the generalized Beltrami equation, provided that the coefficients are compactly supported $VMO$ functions with the expected ellipticity condition, and the weight $\omega$ lies in the Muckenhoupt class $A_p$. As an application, we obtain improved regularity for the jacobian of certain quasiconformal mappings. }

\section{Introduction\label{cha:3}}
\noindent
In this paper, we consider the inhomogeneous, Beltrami equation
\begin{equation} 
\deltabarra f(z)-\mu(z)\,\partial f(z)-\nu(z)\,\overline{\partial f(z)}=g(z),\hspace{1cm}a.e.z\in\C\label{eq:Belnohomogenea}
\end{equation} 
where $\mu$, $\nu$ are $L^\infty(\C;\C)$ functions such that $\||\mu|+|\nu|\|_{\infty}\leq k<1$, and $g$ is a measurable, $\C$-valued function. The derivatives $\partial f,\overline\partial f$ are understood in the distributional sense. In the work \cite{AstIwaSak}, the $L^p$ theory of such equation was developed. More precisely, it was shown that if $1+k<p<1+\frac{1}{k}$ and $g\in L^{p}(\mathbb{C})$ then (\ref{eq:Belnohomogenea}) has a solution $f$, unique modulo constants, whose differential $Df$ belongs to $L^p(\mathbb{C})$, and furthermore, the estimate
\begin{equation}\label{aprioriest}
 \|Df\|_{L^p(\C)}\leq C\,\|g\|_{L^p(\C)}
\end{equation}
holds for some constant $C=C(k,p)>0$. For other values of $p$, \eqref{eq:Belnohomogenea} the claim may fail in general. However, in the previous work \cite{Iwa}, Iwaniec  proved that if $\mu\in VMO(\mathbb{C})$ then for any $1<p<\infty$ and any $g\in L^{p}(\mathbb{C})$ one can find exactly one solution $f$ to the $\C$-linear equation
$$\overline\partial f(z)-\mu(z)\,\partial f(z)=g(z)$$
with $Df\in L^p(\mathbb{C})$. In particular, \eqref{aprioriest} holds whenever $p\in(1,\infty)$. Recently, Koski \cite{Kos} has extended this result to the generalized equation \eqref{eq:Belnohomogenea}. For results in other spaces of functions, see \cite{CruMatOro}. \\
\\
In this paper, we deal with weighted spaces, and so we assume $g\in L^p(\omega)$, $1<p<\infty$. Here $\omega$ is a measurable function, and $\omega>0$ at almost every point. By checking the particular case $\mu=0$, one sees that, for a weighted version of the estimate \eqref{aprioriest} to hold, the Muckenhoupt  condition $\omega\in A_p$ is necessary. It turns out that, for compactly supported $\mu\in VMO$, this condition is also sufficient.

\begin{thm}\label{main}
Let $1<p<\infty$. Let $\mu$ be a compactly supported function in $VMO(\C)$, such that $\|\mu\|_\infty<1$, and let $\omega\in A_p$. Then, the equation
$$\overline\partial f(z)-\mu(z)\,\partial f(z)=g(z)$$
has, for $g\in L^p(\omega)$, a solution $f$ with $Df\in L^p(\omega)$, which is unique up to an additive constant. Moreover, one has
$$\|Df\|_{L^p(\omega)}\leq C\,\|g\|_{L^p(\omega)}$$
for some $C>0$ depending on $\mu$, $p$ and $[\omega]_{A_p}$.
\end{thm}

\noindent
The proof copies the scheme of \cite{Iwa}. In particular, our main tool is the following compactness Theorem, which extends a classical result of Uchiyama \cite{Uchi} about commutators of Calder\'on-Zygmund singular integral operators and $VMO$ functions.

\begin{thm}\label{thm:compacidad}
Let $T$ be a Calder\'on-Zygmund singular integral operator. Let $\omega\in A_{p}$ with $1<p<\infty$, and let $b\in VMO(\mathbb{R}^{n})$. The commutator $[b,T]\colon L^{p}(\omega)\to L^{p}(\omega)$ is compact.
\end{thm}
\noindent
Theorem \ref{thm:compacidad} is obtained from a sufficient condition for compactness in $L^p(\omega)$. When $\omega=1$, this sufficient condition reduces to the classical Frechet-Kolmogorov compactness criterion. Theorem \ref{main} is then obtained from Theorem \ref{thm:compacidad} by letting $T$ be the Beurling-Ahlfors singular integral operator. \\
\\
A counterpart to Theorem \ref{main} for the \emph{generalized Beltrami equation},
\begin{equation}\label{generalized}
\overline\partial f(z)-\mu(z)\,\partial f(z)-\nu(z)\,\overline{\partial f(z)}=g(z), 
\end{equation}
can also be obtained under the ellipticity condition $\||\mu|+|\nu|\|_\infty\leq k<1$ and the $VMO$ smoothness of the coefficients (see Theorem \ref{conjug} below). Theorem \ref{thm:compacidad} is again the main ingredient. However, for \eqref{generalized} the argument in Theorem \ref{main} needs to be modified, because the involved operators are not $\C$-linear, but only $\R$-linear. In other words, $\C$-linearity is not essential. See also \cite{Kos}. \\
\\
It turns out that any linear, elliptic, divergence type equation can be reduced to equation \eqref{generalized} (see e.g.\cite[Theorem 16.1.6]{AsIwMa}). Therefore the following result is no surprise. 

\begin{coro}
Let $K\geq 1$. Let $A:\R^{2}\to\R^{2\times 2}$ be a matrix-valued function, satisfying the ellipticity condition
$$\frac{1}{K}\leq v^t\,A(z)\,v\leq K,\hspace{2cm}\text{ whenever }v\in \R^2,\,|v|=1$$
at almost every point $z\in \R^2$, and such that $A-\Id$ has compactly supported $VMO$ entries. Let $p\in (1,\infty)$ be fixed, and $\omega\in A_p$. For any $g\in L^p(\omega)$, the equation
$$\div (A(z)\,\nabla u)=\div (g)$$
has a solution $u$ with $\nabla u\in L^p(\omega)$, unique  up to an additive constant, and such that
$$\|\nabla u\|_{L^p(\omega)}\leq C\,\|g\|_{L^p(\omega)}$$
for some constant $C=C(A,\omega,p)$. 
\end{coro}

\noindent
Other applications of Theorem \ref{main} are found in connection to planar $K$-quasiconformal mappings. Remember that a $W^{1,2}_{loc}$ homeomorphism $\phi:\Omega\to\Omega'$ between domains $\Omega,\Omega'\subset\C$ is called $K$-quasiconformal if 
$$|\overline\partial \phi(z)|\leq \frac{K-1}{K+1}\,|\partial \phi(z)|\hspace{2cm}\text{for }a.e.z\in\Omega.$$
In general, jacobians of $K$-quasiconformal maps are Muckenhoupt weights belonging to the class $A_p$ for any $p>K$ (see \cite[Theorem 14.3.2 ]{AsIwMa}, or also \cite{AstIwaSak}), and this is sharp. As a consequence of Theorem \ref{main}, we obtain the following improvement. 

\begin{coro}
Let $\mu\in VMO$ be compactly supported, such that $\|\mu\|_\infty<1$, and let $\phi$ be the principal solution of 
$$\overline\partial\phi(z)-\mu(z)\,\partial\phi(z)=0.$$ 
Then the jacobian determinant $J(\cdot,\phi^{-1})$ belongs to $A_p$ for every $1<p<\infty$. 
\end{coro}

\noindent
We actually prove that composition with the inverse mapping $\phi^{-1}$ preserves the Muckenhoupt class $A_2$. Then, the above corollary follows by the results of Johnson and Neugebauer in \cite{JohNeu}, where the composition problem in all Muckenhoupt classes is completely solved. \\
\\
%One further application is given by the a priori $L^p(\omega)$ estimates for the \emph{reduced Beltrami equation},
%$$\overline\partial f(z)-\lambda(z)\,\Im(\partial f(z))=g(z),$$
%see Corollary \ref{reducedapriori} below. This equation arises naturally in connection with the Stoilow factorization for the generalized $\R$-linear Beltrami equation, and has recently been object of several investigations (see \cite{AstJaa} and the references therein).\\
%\\
The paper is structured as follows. In Section \ref{FreKol} we prove Theorem \ref{thm:compacidad}. In Section \ref{Est} we prove Theorem \ref{main} and its counterpart for the generalized Beltrami equation. In Section \ref{BMOconn} we study some applications. By $C$ we denote a positive constant that may change at each occurrence. $B(x,r)$ denotes the open ball with center $x$ and radius $r$, and $2B$ means the open ball concentric with $B$ and having double radius.

\section{Compactness}\label{FreKol}

By singular integral operator $T$, we mean a linear operator on $L^p(\mathbb R^n)$ that can be written as
$$
Tf(x)=\int_{\mathbb R^n}f(y)\,K(x,y)\,\mathrm{d}y.
$$
Here $K\colon \mathbb R^n\times\mathbb R^n\setminus\{x=y\}\to \mathbb C$ obeys the bounds 
\begin{enumerate}
\item $|K(x,y)|\leq \frac{C}{|x-y|^n}$,
\item $|K(x,y)-K(x,y')|\leq C\frac{|y-y'|}{|x-y|^{n+1}}\quad\text{whenever }|x-y|\geq2|y-y'|,$
\item $|K(x,y)-K(x',y)|\leq C\frac{|x-x'|}{|x-y|^{n+1}}\quad\text{whenever }|x-y|\geq2|x-x'|.$
\end{enumerate}
Given a singular integral operator $T$, we define the \emph{truncated singular integral } as 
$$
T_{\epsilon}f(x)=\int_{|x-y|\geq\epsilon}K(x,y)f(y)\mathrm{d}y
$$
and the \emph{maximal singular integral} by the relationship 
$$
T_\ast f(x)=\sup_{\epsilon>0}\left|T_{\epsilon}f(x)\right|.
$$
As usually, we denote $\fint_E f(x)\mathrm{d}x=\frac{1}{|E|}\int_E f(x)\,\mathrm{d}x$.  A weight is a function $\omega\in L_{loc}^{1}(\mathbb{R}^{n})$  such that $\omega(x)>0$ almost everywhere. A weight $\omega$ is said to belong to the Muckenhoupt class $A_p$, $1<p<\infty$, if 
$$
[\omega]_{A_{p}}:=\sup\left(\fint_{Q}\omega(x)\mathrm{d}x\right)\left(\fint_{Q}\omega(x)^{-\frac{p'}{p}}\mathrm{d}x\right)^{\frac{p}{p'}}<\infty,$$
where the supremum is taken over all cubes $Q\subset\mathbb{R}^{n}$, and where $\frac{1}{p}+\frac{1}{p'}=1$. By $L^p(\omega)$ we denote the set of measurable functions $f$ that satisfy 
\begin{equation}
\|f\|_{L^p(\omega)}=\left(\int_{\mathbb{R}^{n}}|f(x)|^{p}\omega(x)\mathrm{d}x\right)^{\frac{1}{p}}<\infty.\label{eq:tat}
\end{equation}
The quantity  $\|f\|_{L^{p}(\omega)}$ defines a complete norm in  $L^{p}(\omega)$. It is well know that if $T$ is a Calder\'on-Zygmund operator, then $T$ and also $T_\ast$ are bounded in $L^{p}(\omega)$ whenever $\omega\in A_p$  (see for instance \cite[Cap. IV, Theorems 3.1 and 3.6]{GarRub}). Also the Hardy-Littlewood maximal operator $M$ is bounded in $L^p$. For more about $A_p$ classes and weighted spaces $L^p(\omega)$, we refer the reader to \cite{GarRub}. \\
\\
We first show the following sufficient condition for compactness in $L^{p}(\omega)$, $\omega\in A_p$. Remember that an metric space $X$ is \emph{totally bounded} if for every $\epsilon>0$ there exists a finite number of open balls of radius $\epsilon$ whose union is the space $X$. In addition, a metric space is compact if and only if is complete and totally bounded. 

\begin{thm}\label{thm:FreKol}

Let $p\in(1,\infty)$, $\omega\in A_{p}$, and let $\mathfrak{F\subset}L^{p}(\omega)$. Then $\mathfrak{F}$ is totally bounded if it satisfies the next three conditions: 
\begin{enumerate}
\item $\mathfrak{F}$ is uniformly bounded, i.e. $\sup_{f\in\mathfrak{F}}\|f\|_{L^p(\omega)}<\infty$. \label{unif} 
\item $\mathfrak{F}$ is uniformly equicontinuous, i.e.\label{equic}
$\sup_{f\in\mathfrak{F}}\|f(\cdot+h)-f(\cdot)\|_{L^{p}(\omega)}\xrightarrow{h\to0}0.$
\item $\mathfrak{F}$ uniformly vanishes at infinity,\label{decae} i.e. $\sup_{f\in\mathfrak{F}}\|f-\chi_{Q(0,R)}f\|_{L^{p}(\omega)}\xrightarrow{R\to\infty}0$, where $Q(0,R)$ is the cube with center at the origin and sidelength $R$. 
\end{enumerate}
\end{thm}

\noindent
Let us emphasize that Theorem \ref{thm:FreKol} is a strong sufficient condition for compactness in $L^p(\omega)$, because for a general weight $\omega\in A_p$ the space $L^p(\omega)$ is not invariant under translations. Theorem \ref{thm:FreKol} is proved by adapting the arguments in \cite{HanHol}. In particular, the following result (which can be found in \cite[Lemma 1]{HanHol}) is essential.

\begin{lema}
Let $X$ be a metric space. Suppose that for every $\epsilon>0$ one can find a number $\delta>0$, a metric space $W$ and an application $\Phi\colon X\to W$ such that $\Phi(X)$ is totally bounded, and the implication
$$d(\Phi(x),\Phi(y))<\delta\hspace{1cm}\Longrightarrow\hspace{1cm}d(x,y)<\epsilon$$
holds for any $x,y\in X$. Then $X$ is totally bounded.
\end{lema}

\begin{proof}[Proof of Theorem \ref{thm:FreKol}]
Suppose that the family $\mathfrak{F}$ satisfies the three conditions of the Theorem \ref{thm:FreKol}. Given $\rho\geq0$, let $Q$ be the largest open cube centered at $0$ such that $2Q\subset B(0,\rho)$.
For $R>0$, let $Q_{1},\ldots,Q_{N}$ be $N$ copies of $Q$ such that have not a overlap and such that
\[
\overline{Q(0,R)}=\overline{\bigcup_{i}Q_{i}},
\]
where $Q(0,R)$ is the open cube on the origin of side $R$. Let us define an application $f\mapsto \Phi f$ by setting
\begin{equation}\label{defP}
\Phi f(x)  =  \begin{cases}
\displaystyle\fint_{Q_{i}}f(z)\mathrm{d}z, & x\in Q_{i},i=1,\ldots,N,%\nonumber
\\
0 & \text{otherwise.}\end{cases}\\
 %& =  \sum_{i=1}^{N}\chi_{Q_{i}}(x)\frac{1}{|Q_{i}|}\int_{Q_{i}}f(z)\mathrm{d}z.
 \end{equation}
For $f\in L^{p}(\omega)$ one has $f\in L_{loc}^{1}(\mathbb{R}^{n})$, and thus $\Phi f$ is well defined for $f\in \mathfrak{F}$. Moreover, 
$$
\aligned
\int_{\mathbb{R}^{n}}|\Phi f(x)|^{p}\omega(x)\mathrm{d}x
&=\sum_{i=1}^N \left|\fint_{Q_i}f(z)\mathrm{d}z\right|^p\,\int_{Q_i}\omega(x)\,\mathrm{d}(x)\\
%&=\int_{\cup_{i=1}^{n}Q_{i}}\left|\frac{1}{|Q_{i}|}\int_{Q_{i}}f(z)\mathrm{d}z\right|^{p}\omega(x)\mathrm{d}x\nonumber \\
%&= \sum_{i=1}^{N}\frac{1}{|Q_{i}|^{p}}\left(\int_{Q_{i}}\omega(x)\mathrm{d}x\right)\left|\int_{Q_{i}}\frac{f(z)\omega^{\frac{1}{p}}(z)}{\omega^{\frac{1}{p}}(z)}\mathrm{d}z\right|^{p}\nonumber \\
& \leq  \sum_{i=1}^{N}\left(\fint_{Q_{i}}|f(z)|^{p}\omega(z)\mathrm{d}z\right)\left(\fint_{Q_{i}}\omega^{\frac{-p'}{p}}(z)\mathrm{d}z\right)^{\frac{p}{p'}} \int_{Q_{i}}\omega(x)\mathrm{d}x  \\
& \leq  [\omega]_{A_{p}}\|f\|_{L^{p}(\omega)}^{p}.
\endaligned
$$
In particular, $\Phi:L^p(\omega)\to L^p(\omega)$ is a bounded operator. As $\mathfrak{F}$ is bounded, then $\Phi(\mathfrak{F})$ is a bounded subset of a finite dimensional Banach space, and hence $\Phi(\mathfrak{F})$ is totally bounded. \\%Now, if $f\in\mathfrak{F}$,  
 %\begin{eqnarray*}
%\|f-Pf\|_{L^{p}(\omega)} & \leq & \|f-f\chi_{Q(0,R)}\|_{L^{p}(\omega)}+\|f\chi_{Q(0,R)}-Pf\|_{L^{p}(\omega)}.
%\end{eqnarray*}
By the vanishing condition at infinity, given $\epsilon>0$ there exists $R_0>0$ such that 
\begin{equation}
\sup_{f\in\mathfrak{F}}\|f-\chi_{Q(0,R)}f\|_{L^{p}(\omega)}<\frac{\epsilon}{4},\qquad\text{if }R>R_{0}.\label{eq:nec22}
\end{equation}
On the other hand, by Jensen's inequality, 
\begin{eqnarray*}
\|f\chi_{Q(0,R)}-\Phi f\|_{L^{p}(\omega)}^{p}& = & \sum_{i=1}^N\int_{Q_{i}}\left|f(x)-\fint_{Q_{i}}f(z)\mathrm{d}z\right|^{p}\omega(x)\mathrm{d}x\\
% & = & \sum_{i=1}^{N}\int_{Q_{i}}\left|\frac{1}{|Q_{i}|}\int_{Q_{i}}(f(x)-f(z))\mathrm{d}z\right|^{p}\omega(x)\mathrm{d}x\\
& \leq & \sum_{i=1}^{N}\frac{1}{|Q_i|}\int_{Q_{i}}\int_{Q_{i}}|f(x)-f(z)|^{p}\mathrm{d}z\:\omega(x)\mathrm{d}x
 \end{eqnarray*}
Now, if $x,z\in Q_{i}$, then $z-x=h\in2Q\subset B(0,\rho)$. Therefore, after a change of coordinates,
$$
\aligned
\|f\chi_{Q(0,R)}-\Phi f\|_{L^{p}(\omega)}^{p} 
& \leq \sum_{i=1}^{N}\frac{1}{|Q_i|}\int_{Q_{i}}\int_{2Q}|f(x)-f(x+h)|^{p}\mathrm{d}h\:\omega(x)\mathrm{d}x\nonumber \\
& =  \frac{1}{|Q|}\int_{2Q}\sum_{i=1}^{N} \int_{Q_{i}}|f(x)-f(x+h)|^{p}\omega(x)\mathrm{d}x\mathrm{\: d}h\nonumber \\
& \leq\frac{1}{|Q|}\int_{2Q}\int_{\mathbb{R}^{n}}|f(x)-f(x+h)|^{p}\omega(x)\mathrm{d}x\:\mathrm{d}h\nonumber 
 %& \leq & \sup_{h\in2Q}\|f(\cdot)-f(\cdot+h)\|_{L^{p}(\omega)}^{p}\left(\frac{1}{|Q|}\int_{2Q}\mathrm{d}h\right)\nonumber \\
\\&\leq 2^{n}\,\sup_{|h|\leq \rho}\left(\sup_{f\in\mathfrak{F}}\|f(\cdot)-f(\cdot+h)\|_{L^{p}(\omega)}^{p}\right).%\label{eq:nec2}
 \endaligned
 $$
Therefore, by the uniform equicontinuity, we can find $\rho>0$ small enough, such that
 \begin{equation}
\sup_{f\in\mathfrak{F}}\|f\chi_{Q(0,R)}-\Phi f\|_{L^{p}(\omega)}<\frac{\epsilon}{4}.\label{eq:nec23}
\end{equation}
By (\ref{eq:nec22}) and (\ref{eq:nec23}) we have that
$$
\sup_{f\in\mathfrak{F}}\|f-\Phi f\|_{L^{p}(\omega)}<\frac{\epsilon}{2},
$$
whence
\begin{equation}
\|f\|_{L^{p}(\omega)}<\frac{\epsilon}{2}+\|\Phi f\|_{L^{p}(\omega)},\hspace{1cm}\text{ whenever }f\in\mathfrak{F}.\label{eq:per3}
\end{equation}
Since $\Phi$ is linear, this means that
$$\|f-g\|_{L^{p}(\omega)}<\frac{\epsilon}{2}+\|\Phi f-\Phi g\|_{L^{p}(\omega)},\hspace{1cm}\text{ whenever }f,g\in\mathfrak{F}.\label{eq:per3}
$$
Set $\delta=\epsilon/2$. The above inequality says that if $f,g\in \mathfrak{F}$ are such that $d(\Phi f,\Phi g)<\delta$, then $d(f,g)<\epsilon$. By the previous Lemma, it follows that $\mathfrak{F}$ is totally bounded.
\end{proof}
\noindent
In order to prove Theorem \ref{thm:compacidad}, we will first reduce ourselves to smooth symbols $b$. Let us recall that commutators $C_{b}=[b,T]$ with $b\in BMO(\mathbb{R}^{n})$ are continuous in $L^{p}(\omega)$ (\emph{e. g.} Theorem 2.3 in \cite{SegTor}). Moreover, in \cite[Theorem 1]{Per} the following estimate is shown, 
\begin{equation}
\|C_{b}f\|_{L^{p}(\omega)}\leq C\,\|b\|_{*}\,\|M^{2}f\|_{L^{p}(\omega)},\label{eq:Perez}
\end{equation}
where $C$ may depend on $\omega$, but not on $b$. Now, by the boudedness of the Hardy-Littlewood operator $M$ on $L^p(\omega)$, we obtain
$$\|C_{b}f\|_{L^{p}(\omega)}\leq C\,\|b\|_\ast\,\|f\|_{L^p(\omega)}.$$
Since by assumption $b\in VMO(\R^n)$, we can approximate the function $b$ by functions $b_{j}\in {\mathcal C}_{c}^{\infty}(\mathbb{R}^{n})$ in the $BMO$ norm, and thus
$$
\|C_{b}f-C_{b_{j}}f\|_{L^{p}(\omega)}=\|C_{b-b_{j}}f\|_{L^{p}(\omega)}\leq C\,\|b-b_j\|_\ast\,\|f\|_{L^p(\omega)}.
$$
In particular, the commutators with smooth symbol $C_{b_j}$ converge to $C_b$ in the operator norm of $L^p(\omega)$. Therefore it suffices to prove compactness for the commutator with smooth symbol. \\
\\
Another reduction in the proof of Theorem \ref{thm:compacidad} will be made by slightly modifying the singular integral operator $T$. This technique comes from Krantz and Li \cite{KraLi2}. More precisely, for every $\eta>0$ small enough, let us take a continuous function $K^\eta$ defined on $\mathbb{R}^{n}\times\mathbb{R}^{n}$, taking values in $\R$ or $\mathbb{C}$, and such that:
\begin{enumerate}
\item $K^{\eta}(x,y)=K(x,y)$ if $|x-y|\geq\eta$ 
\item $|K^{\eta}(x,y)|\leq\frac{C_{0}}{|x-y|^{n}}$ for $\frac{\eta}{2}<|x-y|<\eta$
\label{enu:-para-2} 
\item $K^{\eta}(x,y)=0$ si $|x-y|\leq\frac{\eta}{2}$ 
\end{enumerate}
where $C_{0}$ is independent of $\eta$. Due to the growth properties of $K$, is not restrictive to suppose that the condition \ref{enu:-para-2}
holds for all $x,y\in\mathbb{R}^{n}$. Now, let
$$T^{\eta}f(x)=\int_{\mathbb{R}^{n}}K^{\eta}(x,y)f(y)\mathrm{d}y,$$
and let us also denote
$$C_{b}^{\eta}f(x)=[b,T^\eta]f(x)=\int_{\mathbb{R}^{n}}(b(x)-b(y))K^{\eta}(x,y)f(y)\mathrm{d}y.$$
We now prove that the commutators $C^\eta_b$ approximate $C_b$ in the operator norm.

\begin{lema}\label{lema2}
Let $b\in{\mathcal C}^1_c(\R^n)$. There exists a constant $C=C(n,C_0)$ such that  
$$|C_b f(x)-C_b^\eta f(x)|\leq C\,\eta\,\|\nabla b\|_\infty\,Mf(x)\hspace{1cm}\text{ almost everywhere,}$$
for every $\eta>0$. As a consequence, 
$$\lim_{\eta\to 0}\|C_b^\eta - C_b\|_{L^p(\omega)\to L^p(\omega)}= 0$$ 
whenever $\omega\in A_p$ and $1<p<\infty$.
\end{lema} 
\begin{proof} Let $f\in L^{p}(\omega)$. For every $x\in\mathbb{R}^{n}$ we have:
$$\aligned
C_{b}f(x)-C_{b}^{\eta}f(x)
&=\int_{|x-y|<\eta}(b(x)-b(y))K(x,y)f(y)\mathrm{d}y-\int_{\frac\eta2\leq|x-y|<\eta}(b(x)-b(y))K^{\eta}(x,y)f(y)\mathrm{d}y\\
&= I_{1}(x)+I_{2}(x).
\endaligned$$ 
Using the smoothness of $b$ and the size estimates of $K^\eta$, we have that
$$\aligned
|I_{1}(x)|
&\leq\int_{|x-y|<\eta}|b(y)-b(x)||K(x,y)||f(y)|\mathrm{d}y
%&\leq C_{0}\|\nabla b\|_{\infty} \int_{|x-y|<\eta}\frac{|f(y)|}{|x-y|^{n-1}}\mathrm{d}y\\
\leq C_{0}\,\|\nabla b\|_{\infty}\sum_{j=0}^{\infty}\int_{\frac{\eta}{2^{j+1}}<|x-y|<\frac{\eta}{2^{j}}}\frac{|f(y)|}{|x-y|^{n-1}}\mathrm{d}y\\
%&=  2^{n}C_{0}\:\eta\:\|\nabla b\|_{\infty}\sum_{j=0}^{\infty}\frac{1}{2^{j+1}}\frac{1}{\left(\frac{\eta}{2^{j}}\right)^{n}}\int_{\frac{\eta}{2^{j+1}}<|x-y|<\frac{\eta}{2^{j}}}|f(y)|\mathrm{d}y\\
&\leq 2^{n}C_{0}\,\|\nabla b\|_{\infty}\sum_{j=0}^{\infty}\frac{\eta\,|B(0,1)|}{2^{j+1}} \fint_{|x-y|<\frac{\eta}{2^{j}}}|f(y)|\mathrm{d}y
\leq \eta\,2^{n}\,C_{0}\,\|\nabla b\|_{\infty}\,|B(0,1)|\, Mf(x)
\endaligned
$$ 
for almost every $x$. 
%The punctual estimation implies that
%\begin{eqnarray}
%\|I_{1}\|_{L^{p}(\omega)} & \leq & 2^{n}C_{0}\:\eta\:|B(0,1)|\sum_{j=0}^{\infty}\frac{1}{2^{j+1}}\|\nabla b\|_{\infty}\|Mf\|_{L^{p}(\omega)}\nonumber \\
 %& \leq & 2^{n}C_{0}\:\eta\:|B(0,1)|\sum_{j=0}^{\infty}\frac{1}{2^{j+1}}\:\|\nabla b\|_{\infty}\|f\|_{L^{p}(\omega)}.\label{eq:re1}
 %\end{eqnarray}
For the other term, similarly
$$
\aligned
|I_{2}(x)| 
& \leq \eta\,\|\nabla b\|_\infty\,\int_{\frac{\eta}{2}<|x-y|<\eta}|K^{\eta}(x,y)|\,|f(y)|\,\mathrm{d}y
\leq \eta\,C_{0}\,\|\nabla b\|_\infty\,\int_{\frac{\eta}{2}<|x-y|<\eta}\frac{|f(y)|}{|x-y|^{n}}\mathrm{d}y\\
 & \leq   \eta\,2^n\,C_{0}\,\|\nabla b\|_\infty\,|B(0,1)|\,\fint_{|x-y|<\eta} |f(y)|\,\mathrm{d}y
\leq\eta\,2^n\,C_{0}\,\|\nabla b\|_\infty\,|B(0,1)|\,Mf(x).
 \endaligned
 $$
Therefore, the pointwise estimate follows. Now, the boundedness of $M$ in $L^p(\omega)$ for any $A_p$ weight $\omega$ implies that
$$\aligned
\|C_b f-C_b^\eta f\|_{L^p(\omega)}
&\leq C\,\eta\,\|\nabla b\|_\infty\,\|Mf\|_{L^p(\omega)}\\
&\leq C\,\eta\,\|\nabla b\|_\infty\,\|f\|_{L^p(\omega)}\to 0,\hspace{1cm}\text{ as }\eta \to 0.
\endaligned$$
This finishes the proof of Lemma \ref{lema2}.\end{proof} 
\noindent
We are now ready to conclude the proof of Theorem \ref{thm:compacidad}. From now on, $\eta>0$ and $b\in{\mathcal C}^1_c(\R^n)$ are fixed, and we have to prove that the commutator $C_b^\eta=[b,T^\eta]$ is compact. Thus, the constants that will appear may depend on $b$ and $\eta$.  \\
We denote $\mathfrak{F}=\{C_{b}^{\eta}f; f\in L^p(\omega), \|f\|_{L^p(\omega)}\leq 1\}$. Then $\mathfrak{F}$ is uniformly bounded, because $C_b^\eta$ is a bounded operator on $L^p(\omega)$. To prove the uniform equicontinuity of $\mathfrak{F}$, we must see that
$$
\lim_{h\to 0}\sup_{f\in\mathfrak{F}}\|C_{b}^{\eta}f(\cdot)-C_{b}^{\eta}f(\cdot+h)\|_{L^{p}(\omega)}=0.
$$
To do this, let us write
$$
\aligned
C_{b}^{\eta}&f(x)-C_{b}^{\eta}f(x+h) \\
&= (b(x)-b(x+h))\int_{\mathbb{R}^{n}}K^{\eta}(x,y)f(y)\mathrm{d}y +\int_{\mathbb{R}^{n}}(b(x+h)-b(y))(K^{\eta}(x,y)-K^{\eta}(x+h,y))f(y)\mathrm{d}y\\
 & =  \int_{\R^n}I_{1}(x,y,h)\mathrm{d}y+\int_{\R^n}I_{2}(x,y,h)\mathrm{d}y.
 \endaligned
 $$
 For $I_{1}(x,y,h)$, using the regularity of the function $b$ and the definition of the operator $T_\ast$,
$$
\aligned
\left|\int_{\R^n}I_{1}(x,y,h)\mathrm{d}y\right| 
%& \leq & |b(x)-b(x+h)|\left|\int_{|x-y|>\frac{\eta}{2}}K^{\eta}(x,y)f(y)\mathrm{d}y\right|\\
& \leq   \|\nabla b\|_{\infty}|h|\left|\int_{|x-y|>\frac{\eta}{2}}\left(K^{\eta}(x,y)-K(x,y)\right)f(y)\mathrm{d}y+\int_{|x-y|>\frac{\eta}{2}}K(x,y)f(y)\mathrm{d}y\right|\\
& \leq  \|\nabla b\|_{\infty}|h|\,\left(\int_{|x-y|>\frac{\eta}{2}}|K^{\eta}(x,y)-K(x,y)|\,|f(y)|\mathrm{d}y+T_\ast f(x)\right)\\
& \leq  \|\nabla b\|_{\infty}|h|\,\left(C\,Mf(x)+T_\ast f(x)\right)
\endaligned
$$
for some constant $C>0$ that may depend on $\eta$, but not on $h$. Therefore, by ,
 \begin{equation}
\int\left|\int_{\R^n}I_{1}(x,y,h)\mathrm{d}y\right|^p\omega(x)\mathrm{d}x\leq
C\,|h|\,\|f\|_{L^{p}(\omega)},\label{eq:comp1}\end{equation}
for $C$ independent of $f$ and $h$. Here we used the boundedness of $M$ and $T_\ast$ on $L^p(\omega)$ (see \cite[Chap. IV, Th. 3.6]{GarRub}). We will divide the integral of $I_{2}(x,y,h)$ into three regions:
\begin{eqnarray*}
A & = & \left\{y\in\R^n: |x-y|>\frac{\eta}{2},\quad|x+h-y|>\frac{\eta}{2}\right\} ,\\
B & = & \left\{y\in\R^n: |x-y|>\frac{\eta}{2},\quad|x+h-y|<\frac{\eta}{2}\right\} ,\\
C & = & \left\{y\in\R^n: |x-y|<\frac{\eta}{2},\quad|x+h-y|>\frac{\eta}{2}\right\} .
\end{eqnarray*}
Note that $I_2(x,y,h)=0$ for $y\in\R^n\setminus A\cup B \cup C$. Now, for the integral over $A$, we use the smoothness of $b$ and $K^\eta$,
$$
\aligned
 \left|\int_{A}I_2(x,y,h)\mathrm{d}y\right|  
 &\leq   C \|\nabla b\|_{\infty}|h|\int_{|x-y|>\frac{\eta}{4}}\frac{|f(y)|}{|x-y|^{n+1}}\mathrm{d}y \\
 %& \leq   C \|\nabla b\|_{\infty}|h|\sum_{j=0}^{\infty}\frac{1}{\left(2^{j+1}\frac{\eta}{4}\right)^{n+1}}\int_{|x-y|<\frac{2^{j}\eta}{4}}|f(y)|\mathrm{d}y\nonumber \\
 & \leq  C \|\nabla b\|_{\infty}\frac{|h|}{\eta}\,\sum_{j=0}^{\infty}2^{-j}\fint_{|x-y|<\frac{2^{j}\eta}{4}}|f(y)|\mathrm{d}y
\leq   C \|\nabla b\|_{\infty}\frac{|h|}{\eta}\quad Mf(x),
\endaligned$$
thus
$$
\int_{\R^n}\left|\int_A I_2(x,y,h)\mathrm{d}y\right|^p\omega(x)\mathrm{d}x\leq C |h|\,\|f\|_{L^p(\omega)}.
$$
for some constant $C$ that may depend on $\eta$, but not on $h$. In particular, the term on the right hand side goes to $0$ as $|h|\to 0$.\\
The integrals of $I_{2}(x,y,h)$ over $B$ and $C$ are symmetric, so we only give the details once. For the integral over the set $B$, let us assume that $|h|$ is very small. We can first choose $R_0>\eta/2+|h|$ such that $b$ vanishes outside the ball $B_0=B(0,R_0)$. It then follows that $b(\cdot + h)$ has support in $2B_0$. Then, since $B\subset B(x,|h|+\eta/2)$, we have for $|x|<3R_0$ that $B\subset 4B_0$ and therefore
$$
\aligned
\left|\int_B I_{2}(x,y,h)\mathrm{d}y\right| 
& \leq  C_{0} \,\|\nabla b\|_\infty\int_{B\cap 4B_{0}}\frac{|x+h-y|\,|f(y)|}{|x-y|^{n}}\mathrm{d}y \leq  C_{0} \,\|\nabla b\|_\infty\int_{B\cap 4B_{0}}\frac{ |f(y)|}{|x-y|^{n-1}}\mathrm{d}y\\
&\leq C_{0} \,\|\nabla b\|_\infty\,(2/\eta)^{n-1}\int_{B\cap 4B_{0}}  |f(y)| \omega(y)^\frac1p \omega(y)^{-\frac1p} \mathrm{d}y\\
&\leq C_{0} \,\|\nabla b\|_\infty\,(2/\eta)^{n-1} \|f\|_{L^p(\omega)}\,\left(\int_{B\cap 4B_0}\omega(y)^{-\frac{p'}{p}}\,\mathrm{d}y\right)^\frac{1}{p'}
%
%\int_{B\cap 3B_{0}}  |f(y)| \omega(y)^\frac1p \omega(y)^{-\frac1p} \mathrm{d}y\\
%
%&\leq \left(\frac{2}{\eta}\right)^{n-1}C_{0}\|\nabla b\|_{\infty}\|f\|_{L^{p}(\omega)}\left(\int_{2B_0}\omega(y)^{-\frac{p'}{p}}\mathrm{d}y\right)^{\frac{1}{p'}}
 \endaligned
 $$
whence
$$\aligned
\int_{3B_0}\left|\int_B I_{2}(x,y,h)\mathrm{d}y\right|^{p}\omega(x) \mathrm{d}x
& \leq  
C\,\|f\|_{L^{p}(\omega)}^p\,\left(\int_{3B_0}\omega(x)\mathrm{d}x\right)\left(\int_{B\cap 4B_{0}}\omega(y)^{-\frac{p'}{p}}\mathrm{d}y\right)^{\frac{p}{p'}} 
%\\
 %& \leq  C\,[\omega]_{A_p}|4B_0|^p \,\|f\|_{L^{p}(\omega)}^{p} ,
\endaligned$$
for some constant $C$ that might depend on $\eta$, but not on $h$. % Also
%$$
%\int_{|x|<3R_{0}}|I_{2}(x,h)|_{B}|^{p}\omega(x)\mathrm{d}x\to0
%$$
%when $|B|\to0$ of uniformly in $f$ . Now, i
If, instead, we have $|x|\geq 3R_{0}$, then $b(x+h)=0$ (because $|h|<R_0$ so that $|x+h|>2R_0$). Note also that for $y\in B$ one has $|x|\leq C |x-y|$ where $C$ depends only on $\eta$. Therefore
$$
\aligned
\left|\int_B I_{2}(x,y,h)\mathrm{d}y\right|  \leq   C\|b\|_{\infty}\int_{B\cap 4B_0}\frac{|f(y)|}{|x-y|^{n}}\mathrm{d}y 
& \leq \frac{C \|b\|_{\infty}}{|x|^{n}}\int_{B\cap 4B_0}|f(y)|\mathrm{d}y\\
&%\leq \frac{C \|b\|_{\infty}}{|x|^{n}}\int_{B\cap 2B_0}\frac{|f(y)|\omega(y)^{\frac{1}{p}}}{\omega(y)^{\frac{1}{p}}}\mathrm{d}y 
  \leq \frac{C \|b\|_{\infty}}{|x|^{n}}\, \|f\|_{L^{p}(\omega)} \left(\int_{B\cap 4B_0}\omega(y)^{-\frac{p'}{p}}\mathrm{d}y\right)^{\frac{1}{p'}}.
  \endaligned
  $$
This implies that 
$$\aligned
\int_{\R^n\setminus 3B_0}\left|\int_{B}I_{2}(x,y,h)\mathrm{d}y\right|^{p}\omega(x)\mathrm{d}x 
& \leq C \|b\|_{\infty}^p\,\|f\|_{L^{p}(\omega)}^{p}\left(\int_{\R^n\setminus 3B_0}\frac{\omega(x)}{|x|^{np}}\mathrm{d}x\right)\left(\int_{B\cap 4B_0}\omega(y)^{-\frac{p'}{p}}\mathrm{d}y\right)^{\frac{p}{p'}}  \\
 %& \leq  C_{3}\left(\int_{B\cap B(0,3R_{0})}\omega(y)^{-\frac{p'}{p}}\mathrm{d}y\right)^{\frac{p}{p'}}\xrightarrow[unif]{h\to0}0
\endaligned$$
Summarizing,
\begin{equation}\label{eq22}
\aligned
\int_{\R^n}&\left|\int_B I_2(x,y,h)\mathrm{d}y\right|^p\omega(x)\,\mathrm{d}x\\
&\leq 
C\,\|f\|_{L^p(\omega)}\,\left(\int_{B\cap 4B_0}\omega(y)^{-\frac{p'}{p}}\mathrm{d}y\right)^{\frac{p}{p'}}\,\left(\int_{3B_0}\omega(x)\,\mathrm{d}x+\int_{\R^n\setminus 3B_0}\frac{\omega(x)}{|x|^{np}}\mathrm{d}x\right) 
\endaligned
\end{equation}
After proving that
$$\int_{|x|>3R_{0}}\frac{\omega(x)}{|x|^{np}}\mathrm{d}x<\infty$$
the left hand side of \eqref{eq22} will converge to $0$ as $|h|\to 0$ since $|B|\to 0$ as $|h|\to 0$. To prove the above claim, let us choose $q<p$ such that $\omega\in A_{q}$  \cite[Theorem 2.6, Ch. IV]{GarRub}. For such $q$, we have
$$
\int_{|x|>R}\frac{\omega(x)}{|x|^{np}}\mathrm{d}x =\sum_{j=1}^{\infty}\int_{2^{j-1}<\frac{|x|}{R}<2^{j}}\frac{\omega(x)}{|x|^{np}}\mathrm{d}x\leq  \sum_{j=1}^{\infty}(2^{j-1}R )^{-np}\omega(B(0,2^jR ))$$
By \cite[Lemma 2.2]{GarRub}, we have
\begin{equation}\label{eq:deca1}
\int_{|x|>R}\frac{\omega(x)}{|x|^{np}}\mathrm{d}x \leq \sum_{j=1}^{\infty}(2^{j-1}R)^{-np}(2^{j}R )^{nq}\omega(B(0,1))=\frac{C}{R^{n(p-q)}}<\infty\end{equation}
as desired. The equicontinuity of $\mathfrak{F}$ follows. \\
Finally, we show the decay at infinity of the elements of $\mathfrak{F}$ . Let $x$ be such that $|x|>R>R_0$. Then, $x\not\in\mathrm{supp}\, b$, and
$$\aligned
|C_{b}^{\eta}f(x)| 
& = \left|\int_{\mathbb{R}^{n}}(b(x)-b(y))K^{\eta}(x,y)f(y)\mathrm{d}y\right| \\
&\leq  C_{0}\|b\|_{\infty}\int_{\mathrm{supp}\, b}\frac{|f(y)|}{|x-y|^{n}}\mathrm{d}y\\
 &\leq \frac{C\| b\|_{\infty}}{|x|^{n}}\int_{\mathrm{supp}\, b} |f(y)|\,\mathrm{d}y \\
 &\leq \frac{C \|b\|_{\infty}}{|x|^{n}}\,\|f\|_{L^{p}(\omega)}\,\left(\int_{\mathrm{supp}\, b}\omega(y)^{-\frac{p'}{p}}dy\right)^{\frac{1}{p'}} 
\endaligned$$
whence
$$
\left(\int_{|x|>R }|C_{b}^{\eta}f(x)|^{p}\omega(x)\mathrm{d}x\right)^\frac1p\leq C\|b\|_\infty\|f\|_{L^{p}(\omega)}\left(\int_{|x|>R}\frac{\omega(x)}{|x|^{np}}\,\mathrm{d}x\right)^\frac1p.
$$
The right hand side above converges to $0$ as $R\to\infty$, due to \eqref{eq:deca1}. By Theorem \ref{thm:FreKol}, $\mathfrak{F}$ is totally bounded. Theorem \ref{thm:compacidad} follows. 

\section{A priori estimates for Beltrami equations}\label{Est}

We first prove Theorem \ref{main}. To do this, let us remember that the Beurling-Ahlfors singular integral operator is defined by the following principal value
$$\B f(z)=-\frac1\pi\,P.V.\int \frac{f(w)}{(z-w)^2}\,\mathrm{d}w.$$
This operator can be seen as the formal $\partial$ derivative of the Cauchy transform,
$${\mathcal C} f(z)=\frac{1}{\pi}\int \frac{f(w)}{z-w}\,\mathrm{d}w.$$
At the frequency side, $\B$ corresponds to the Fourier multiplier $m(\xi)=\frac{\bar{\xi}}{\xi}$, so that $\B$ is an isometry in $L^2(\C)$. Moreover, this Fourier representation also explains the important relation 
$$\B(\overline\partial f)=\partial f$$
for smooth enough functions $f$.  By $\B^\ast$ we mean the singular integral operator obtained by simply conjugating the kernel of $\B$, that is,
$$\B^\ast(f)(z)=-\frac{1}{\pi}\,P.V.\int\frac{ f(w) }{(\bar{z}-\bar{w})^2}\,\mathrm{d}w.$$
Note that $\B^\ast$ has Fourier multiplier $m^\ast(\xi)=\frac{\xi}{\bar{\xi}}$. Thus, 
$$\B\B^\ast=\B^\ast\B =\Id.$$ 
In other words, $\B^\ast$ is the $L^2$-inverse  of $\B$. It also appears as the $\C$-linear adjoint of $\B$, 
$$\int_\C \B f (z)\,\overline{g(z)}\,\mathrm{d}z=\int_\C f(z)\,\overline{\B^\ast g(z)}\,\mathrm{d}z.$$
The complex conjugate operator $\overline\B$ is the composition of $\B$ with the complex conjugation operator $\mathbf{C}f=\overline{f}$, that is, 
$$\overline{\B}(f)=\mathbf{C} \B(f)=\overline{\B (f)}.$$
It then follows that
$$\overline{\B}=\mathbf{C} \B=\B^\ast \mathbf{C}.$$
Note that $\B$ and $\B^\ast$ are $\C$-linear operators, while $\overline\B$ is only $\R$-linear. See \cite[Chapter 4]{AsIwMa} for more about the Beurling-Ahlfors transform.

\begin{proof}[Proof of  Theorem \ref{main}]
We follow Iwaniec's idea \cite[pag. 42--43]{Iwa}. For every $N=1,2,...$, let  
$$P_{N}=\Id+\mu \B+\cdots+(\mu \B)^{N}.$$
Then
\begin{equation}\label{fredh}
(\Id-\mu \B)P_{N-1}=P_{N-1}(\Id-\mu \B)=\Id-\mu^{N}\B^{N}+K_N 
\end{equation}
where $K_N=\mu^{N}\B^{N}-(\mu \B)^{N}$. Each $K_N$ consists of a finite sum of operators that contain the commutator $[\mu,\B]$ as a factor. Thus, by Theorem \ref{thm:compacidad}, each $K_N$ is compact in $L^p(\omega)$. On the other hand, the iterates of the Beurling transform $\B^N$ have the kernel
$$b_{N}(z)=\frac{(-1)^{N}N}{\pi}\frac{\bar{z}^{N-1}}{z^{N+1}}.$$
Therefore,
$$\|\B^{N}\|_{L^{p}(\omega)} \leq  CN^{2},$$
where the constant $C$ depends on $[\omega]_{A_{p}}$, but not on $N$. As a consequence, 
$$\|\mu^{N}\B^{N}f\|_{L^{p}(\omega)}\leq CN^{2}\|\mu\|_{\infty}^{N}\|f\|_{L^{p}(\omega)},$$
and therefore, for large enough $N$, the operator  $\Id-\mu^{N}\B^N$ is invertible. This, together with \eqref{fredh}, says that $\Id-\mu \B$ is an Fredholm operator. Now apply the index theory to $\Id-\mu\B$. The continuous deformation $\Id-t\mu \B$, $0\leq t\leq 1$, is a homotopy from the identity operator to $\Id-\mu\B$. By the homotopical invariance of $\ind$, 
$$
\ind(\Id-\mu \B)=\ind(\Id)=0.
$$
Since injective operators with $0$ index are onto, for the invertibility of $\Id-\mu\B$ it just remains to show that it is injective.  So let $f\in L^p(\omega)$ be such that $f=\mu \B f$. Then $f$ has compact support. Now, since belonging to $A_p$ is an open-ended condition (see e.g. \cite[Theorem IV.2.6]{GarRub}), there exists $\delta>0$ such that $p-\delta>1$ and $\omega\in A_{p-\delta}$. Then $\omega^{-\frac{1}{p-\delta}}\in L_{loc}^{1}(\mathbb{C})$. Taking $\epsilon=\frac{\delta}{p-\delta}$, we obtain
\begin{eqnarray*}
\int_{\mathbb{C}}|f(x)|^{1+\epsilon}\mathrm{d}x & \leq & \left(\int_{\supp f}|f(x)|^{p}\omega(x)\mathrm{d}x\right)^{\frac{1+\epsilon}{p}}\left(\int_{\supp f}\omega(x)^{-\frac{1+\epsilon}{p-(1+\epsilon)}}\mathrm{d}x\right)^{\frac{p-(1+\epsilon)}{p}}\\
 & \leq & \|f\|_{L^{p}(\omega)}^{1+\epsilon}\left(\int_{\supp f}\omega(x)^{-\frac{1+\epsilon}{p-(1+\epsilon)}}\mathrm{d}x\right)^{\frac{p-(1+\epsilon)}{p}}<\infty,
 \end{eqnarray*}
therefore $f\in L^{1+\epsilon}(\mathbb{C})$. But $\Id-\mu\B$ is injective on $L^p(\mathbb C)$, $1<p<\infty$ when $\mu\in VMO(\mathbb{C})$, by Iwaniec's Theorem. Hence, $f\equiv0$.\\
Finally, since $\Id-\mu\B:L^p(\omega)\to L^p(\omega)$ is linear, bounded, and invertible, it then follows that it has a bounded inverse, so the inequality
$$\|g\|_{L^p(\omega)}\leq C\,\|(\Id-\mu\B)g\|_{L^p(\omega)}$$
holds for every $g\in L^p(\omega)$. Here the constant $C>0$ depends only on the $L^p(\omega)$ norm of $\Id-\mu\B$, and therefore on $p, k$ and $[\omega]_{A_p}$, but not on $g$. As a consequence, given $g\in L^p(\omega)$, and setting 
$$f:={\mathcal C}(\Id-\mu\B)^{-1}g,$$ 
we immediately see that $f$ satisfies $\overline\partial f-\mu\partial f=g$. Moreover, since $\omega\in A_p$,
$$\aligned
\|Df\|_{L^p(\omega)}&\leq\|\partial f\|_{L^p(\omega)}+\|\overline\partial f\|_{L^p(\omega)}\\
&=\|\B(\Id-\mu\B)^{-1}g\|_{L^p(\omega)}+\|(\Id-\mu\B)^{-1}g\|_{L^p(\omega)}\leq C\|g\|_{L^p(\omega)},\endaligned$$
where still $C$ depends only on $p,k$ and $[\omega]_{A_p}$.\\
For the uniqueness, let us choose two solutions $f_1$, $f_2$ to the inhomogeneous equation. The difference $F= f_1-f_2$ defines a solution to the homogeneous equation $\overline\partial F-\mu\,\partial F=0$. Moreover, one has that $DF\in L^p(\omega)$ and, arguing as before, one sees that $DF\in L^{1+\epsilon}(\C)$. In particular, this says that $(I-\mu \B)(\overline\partial F)=0$. But for $\mu\in VMO(\C)$, it follows from Iwaniec's Theorem that $\Id-\mu\B$ is injective in $L^p(\C)$ for any $1<p<\infty$, whence $\overline\partial F=0$. Thus $DF=0$ and so $F $ is a constant.
\end{proof}

\noindent
The $\C$-linear Beltrami equation is a particular case of the following one,
$$\overline\partial f(z)-\mu(z)\,\partial f(z)-\nu(z)\,\overline{\partial f(z)}=g(z),$$
which we will refer to as the \emph{generalized Beltrami equation}. It is well known that, in the plane, any linear, elliptic system, with two unknowns and two first-order equations on the derivatives, reduces to the above equation (modulo complex conjugation), whence the interest in understanding it is very big. An especially interesting example is obtained by setting $\mu=0$, when one obtains the so-called \emph{conjugate Beltrami equation},
$$\overline\partial f(z)-\nu(z)\,\overline{\partial f(z)}=g(z).$$
%This equation arises naturally in the complex formulation of the isotropic version of Calder\'on's conductivity problem \cite{AstPai}. 
A direct adaptation of the above proof immediately drives the problem towards the commutator $[\nu,\overline{\B}]$. Unfortunately, as an operator from $L^p(\omega)$ onto itself, such commutator is not compact in general, even when $\omega=1$. To show this, let us choose
$$\nu=i\,\nu_0\,\chi_\D + \nu_1\chi_{\C\setminus\D}$$ 
where the constant $\nu_0\in\R$ and the function $\nu_1$ are chosen so that $\nu$ is continuous on $\C$, compactly supported in $2\D$, with $\|\nu\|_\infty<1$. Let us also consider 
$$E=\left\{ f\in L^p;\,\|f\|_{L^p}\leq 1,\,\supp(f)\subset \D\right\},$$
which is a bounded subset of $L^p$. For every $f\in E$, one has
$$\aligned
\nu\,\overline{\B(f)}-\overline{\B(\nu f)}
&=\chi_\D  i \nu_0 \overline{\B(f)}+\chi_{\C\setminus\D}\,\nu_1\,\overline{\B(f)}-\overline{\B(i \nu_0 f)}\\
&=\chi_\D  i \nu_0 \overline{\B(f)}+\chi_{\C\setminus\D}\,\nu_1\,\overline{\B(f)}+i\nu_0\overline{\B( f)}\\
&=\chi_\D  2i \nu_0 \overline{\B(f)}+\chi_{\C\setminus\D}\,(i\nu_0+\nu_1)\,\overline{\B(f)}.
\endaligned$$ 
In view of this relation, and since $\B$ is not compact, we have just cooked a concrete example of function $\nu\in VMO$ for wich the commutator $[\nu,\overline\B]$ is not compact. 
Nevertheless, it turns out that still a priori estimates hold, even for the generalized equation. 

\begin{thm}\label{conjug}
Let $1<p<\infty$, $\omega\in A_p$, and let $\mu,\nu\in VMO(\C)$ be compactly supported, such that $\||\mu|+|\nu|\|_\infty<1$. Let $g\in L^p(\omega)$. Then the equation
$$\overline\partial f(z)-\mu(z)\,\partial f(z)-\nu(z)\,\overline{\partial f(z)}=g(z)$$
has a solution $f$ with $Df\in L^p(\omega)$ and
$$\|Df\|_{L^p(\omega)}\leq C\,\|g\|_{L^p(\omega)}.$$
This solution is unique, modulo an additive constant.
\end{thm}

\noindent
A previous proof for the above result has been shown in \cite{Kos} for the constant weight $\omega=1$. For the weighted counterpart, the arguments are based on a Neumann series argument similar to that in \cite{Kos}, with some minor modification. We write it here for completeness. The following Lemma will be needed.% Let us remind that $\mathbf{C}$ denotes the complex conjugation operator.

\begin{lema}\label{adjoints}
Let $\mu,\nu\in L^\infty(\C)$ be measurable, bounded with compact support, such that $\||\mu|+|\nu|\|_\infty<1$. If $1<p<\infty$  and $p'=\frac{p}{p-1}$, then the following statements are equivalent:
\begin{enumerate}
\item The operator $\Id-\mu\,\B-\nu\,\overline{\B}:L^p(\C)\to L^p(\C)$ is bijective.
\item The operator $\Id-\overline\mu\,\B^\ast-\nu\,\overline{\B^\ast}:L^{p'}(\C)\to L^{p'}(\C)$ is bijective.
\end{enumerate}
\end{lema}
\begin{proof}
When $\nu=0$, the above result is well known, and follows as an easy consequence of the fact that, with respect to the dual pairing
\begin{equation}\label{complexdual}
\langle f,g\rangle=\int_\C f(z)\,\overline{g(z)}\,\mathrm{d}z,
\end{equation}
the operator $\Id-\mu\B:L^p(\C)\to L^p(\C)$ has precisely $\Id-\B^\ast\overline\mu:L^{p'}(\C)\to L^{p'}(\C)$ as its $\C$-linear adjoint. Unfortunately, when $\nu$ does not identically vanish, $\R$-linear operators do not have an adjoint with respect to this dual pairing. An alternative proof can be found in \cite{Kos}. To this end, we think the space of $\C$-valued $L^p$ functions $L^p(\C)$ as an $\R$-linear space,
$$L^p(\C)=L^p_\R(\C)\oplus L^p_\R(\C),$$
by means of the obvious identification $u+iv=(u,v)$. According to this product structure, every bounded $\R$-linear operator $T:L^p_\R(\C)\oplus L^p_\R(\C)\to L^p_\R(\C)\oplus L^p_\R(\C)$ has an obvious matrix representation
$$T(u+iv)=T\left(\begin{array}{c}u\\ v\end{array}\right)=\left(\begin{array}{cc}T_{11}&T_{12}\\T_{21}&T_{22}\end{array}\right)\left(\begin{array}{c}u\\ v\end{array}\right),$$
where every $T_{ij}:L^p_\R(\C)\to L^p_\R(\C)$ is bounded. Similarly, bounded linear functionals $U: L^p_\R(\C)\oplus L^p_\R(\C)\to \R$ are represented by
$$U\left(\begin{array}{c}u\\ v\end{array}\right)=\left(\begin{array}{cc}U_{1}&U_{2}\end{array}\right)\left(\begin{array}{c}u\\ v\end{array}\right),$$
where every $U_j:L^p_\R(\C)\to\R$ is bounded. By the Riesz Representation Theorem, we get that $L^p_\R(\C)\oplus L^p_\R(\C)$ has precisely $L^{p'}_\R(\C)\oplus L^{p'}_\R(\C)$  as its topological dual space. In fact, we have an $\R$-bilinear dual pairing,
$$
\langle \left(\begin{array}{c}u\\v\end{array}\right),\left(\begin{array}{c}u'\\v'\end{array}\right)\rangle=\int u(z)\,u'(z)\,\mathrm{d}z+\int v(z)\,v'(z)\,\mathrm{d}z,
$$
whenever $(u,v)\in L^p_\R(\C)\oplus L^p_\R(\C)$ and $(u',v')\in L^{p'}_\R(\C)\oplus L^{p'}_\R(\C)$, and 
which is nothing but the real part of \eqref{complexdual}. Under this new dual pairing, every $\R$-linear opeartor $T:L^p_\R(\C)\oplus L^p_\R(\C)\to L^p_\R(\C)\oplus L^p_\R(\C)$ can be associated another operator $$T': L^{p'}_\R(\C)\oplus L^{p'}_\R(\C)\to  L^{p'}_\R(\C)\oplus L^{p'}_\R(\C),$$
called the $\R$-adjoint operator of $T$, defined by the common rule 
$$\langle \left(\begin{array}{c}u\\v\end{array}\right),T'\left(\begin{array}{c}u'\\v'\end{array}\right)\rangle=\langle T\left(\begin{array}{c}u\\v\end{array}\right),\left(\begin{array}{c}u'\\v'\end{array}\right)\rangle.$$ 
If $T$ is a $\C$-linear operator, then $T'$ is the same as the $\C$-adjoint $T^\ast$ (i.e. the adjoint with respect to \eqref{complexdual}) so in particular for the Beurling-Ahlfors transform $\B$ we have an $\R$-adjoint $\B'$, and moreover $\B^\ast=\B'$. Similarly, the pointwise multiplication by $\mu$ and $\nu$ are also $\C$-linear operators. Thus their $\R$-adjoints $\mu'$, $\nu'$ agree with their respectives $\C$-adjoints $\mu^\ast$, $\nu^\ast$. But these are precisely the pointwise multiplication with the respective complex conjugates. Symbollically, $\mu'=\overline\mu$ and $\nu'=\overline{\nu}$. In contrast, general $\R$-linear operators need not have a $\C$-adjoint. For example, for the complex conjugation,
$$\mathbf{C}=\left(
\begin{array}{cc}
\Id& 0\\ 0& -\Id
\end{array}
\right)$$ 
one simply has $\mathbf{C}'=\mathbf{C}$. Putting all these things together, one easily sees that
$$\aligned
(\Id-\mu\B-\nu\overline\B)'
&=(\Id-\mu\B-\nu\mathbf{C}\B)'\\
&=\Id-(\mu\B)'-(\nu\mathbf{C}\B)'\\
&=\Id-\B'\mu'-\B'\mathbf{C}'\nu'\\
&=\Id-\B^\ast\overline\mu-\B^\ast\mathbf{C}\overline\nu\\
%&=\B^\ast\B-\B^\ast\overline\mu\B^\ast\B-\B^\ast\mathbf{C}\overline\nu\B^\ast\B\\
&=\B^\ast\left(\Id-\overline\mu\B^\ast - \mathbf{C}\overline\nu\B^\ast\right)\B\\
&=\B^\ast\left(\Id-\overline\mu\B^\ast - \nu\mathbf{C}\B^\ast\right)\B
\endaligned$$
where we used the fact that $\B^\ast\B=\B\B^\ast=\Id$. As a consequence, and using that both $\B$ and $\B^\ast$ are bijective in $L^p(\C)$, we obtain that the bijectivity of the operator $\Id-\mu\B-\nu\overline\B$ in $L^p_\R(\C)\oplus L^p_\R(\C)$ is equivalent to that of $\Id-\overline\mu\B^\ast - \nu\mathbf{C}\B^\ast$ in the dual space $L^{p'}_\R(\C)\oplus L^{p'}_\R(\C)$. Similarly, one proves that
$$(\Id-\mu\B^\ast-\nu\mathbf{C}\B^\ast)'=\B(\Id-\overline\mu\,\B-\nu\,\overline\B)\B^\ast.$$
Hence, the bijectivity of $\Id-\mu\B^\ast-\nu\mathbf{C}\B^\ast$ in $L^p_\R(\C)\oplus L^p_\R(\C)$   is equivalent to the bijectivity of $\Id-\overline\mu\,\B-\nu\,\overline\B$ in $L^{p'}_\R(\C)\oplus L^{p'}_\R(\C)$.
\end{proof}

\begin{lema}\label{generalfredholm}
If $1<p<\infty$, $\omega\in A_p$, $\mu,\nu\in VMO$ have compact support, and $\||\mu|+|\nu|\|_\infty\leq k<1$, then the operators
$$\Id-\mu\B-\nu\overline\B\hspace{2cm}\text{and}\hspace{2cm}\Id-\mu\B^\ast-\nu\overline{\B^\ast}$$
are Fredholm operators in $L^p(\omega)$.
\end{lema}
\begin{proof}
We will show the claim for the operator $\Id-\mu\B-\nu\overline\B$. For $\Id-\mu\B^\ast-\nu\overline{\B^\ast}$ the proof follows similarly.  It will be more convenient for us to write $\overline\B=\mathbf{C}\B$. As in the proof of Theorem \ref{main}, we set
$$P_N=\sum_{j=0}^N(\mu\B+\nu \mathbf{C}\B)^j.$$
Then
$$\aligned
(\Id-\mu \B-\nu \mathbf{C}\B)\circ P_{N-1}&=\Id-(\mu\B+\nu \mathbf{C}\B)^N,\\
P_{N-1}\circ (\Id-\mu\B+\nu \mathbf{C}\B)&=\Id-(\mu\B+\nu \mathbf{C}\B)^N.\endaligned$$
We will show that
\begin{equation}\label{decomp}
(\mu\B+\nu \mathbf{C}\B)^N=R_N+K_N
\end{equation}
where $K_N$ is a compact operator, and $R_N$ is a bounded, linear operator such that 
$$\|R_N f \|_{L^p(\omega)}\leq C\,k^N\,N^3\,\|f\|_{L^p(\omega)}.$$ 
Then, the Fredholm property follows immediately. To prove \eqref{decomp}, let us write, for any two operators $T_1$, $T_2$,  
$$
(T_1+T_2)^N=\sum_{\sigma\in \{1,2\}^N}T_\sigma,
$$
where $\sigma\in\{1,2\}^N$ means that $\sigma=(\sigma(1),\dots,\sigma(N))$ and $\sigma(j)\in\{1,2\}$ for all $j=1,\dots,N$, and
$$T_\sigma=T_{\sigma(1)} T_{\sigma(2)}\dots T_{\sigma(N)}.$$
By choosing $T_1=\mu \B$ and $T_2=\nu\mathbf{C}\B$, one sees that every $T_{\sigma(j)}$ can be written as
$$T_{\sigma(j)}=M_{\sigma(j)} C_{\sigma(j)} \B$$
being $M_1=\mu$, $M_2=\nu$, $C_1=\Id$ and $C_2=\mathbf{C}$. Thus
$$T_\sigma=M_{\sigma(1)}C_{\sigma(1)} \B \,  M_{\sigma(2)}C_{\sigma(2)} \B\, \dots  M_{\sigma(N)}C_{\sigma(N)} \B.$$
Our main task consists of rewriting $T_\sigma$ as
\begin{equation}\label{rightorder}
T_\sigma=M_{\sigma(1)}C_{\sigma(1)}M_{\sigma(2)}C_{\sigma(2)}   \dots M_{\sigma(N)}C_{\sigma(N)}\,B_\sigma +K_\sigma.
\end{equation}
for some compact operator $K_\sigma$ and some bounded operator $B_\sigma\in\left\{\B,\B^\ast\right\}^N$. If this is possible, then one gets that
$$\aligned
(T_1+T_2)^N
&=\sum_{\sigma\in \{1,2\}^N}M_{\sigma(1)}C_{\sigma(1)}M_{\sigma(2)}C_{\sigma(2)} \dots M_{\sigma(N)}C_{\sigma(N)}\,B_\sigma+\sum_{\sigma\in \{1,2\}^N}K_\sigma\\
&=R_N+K_N.
\endaligned
$$
It is clear that $K_N$ is compact (it is a finite sum of compact operators). Moreover, from $B_\sigma\in\left\{\B,\B^\ast\right\}^N$, one has
$$
|B_\sigma f(z)|\leq \sum_{j=1}^N|\B^nf(z)| + \sum_{j=1}^N|(\B^\ast)^n f(z)|.
$$
Thus
$$\aligned
|R_Nf(z)|
&\leq\sum_{\sigma\in \{1,2\}^N} |M_{\sigma(1)}C_{\sigma(1)}\dots M_{\sigma(N)}C_{\sigma(N)}\,B_\sigma f(z)|\\
&\leq\sum_{\sigma\in \{1,2\}^N} |M_{\sigma(1)}(z)| \dots |M_{\sigma(N)}(z)|\left(\sum_{n=1}^N|\B^nf(z)| + \sum_{j=1}^N|(\B^\ast)^nf(z)|\right)\\
&=\left(\sum_{n=1}^N|\B^nf(z)| + \sum_{j=1}^N|(\B^\ast)^nf(z)|\right).
\,(|M_1(z)|+|M_2(z)|)^N
\endaligned$$
Now, since $\|\B^jf\|_{L^p(\omega)}\leq C_\omega\,j^2\,\|f\|_{L^p(\omega)}$ (and similarly for $(\B^\ast)^n$), one gets that
$$\aligned
\|R_N f\|_{L^p(\omega)}
&\leq \||M_1|+|M_2|\|_\infty^N\,C_\omega \,\left(\sum_{j=1}^N j^2\right)\,\|f\|_{L^p(\omega)}\\
&=C\,k^N\,N^3\,\|f\|_{L^p(\omega)}
\endaligned$$
and so \eqref{decomp} follows from the representation \eqref{rightorder}. To prove that representation \eqref{rightorder} can be found, we need the help of Theorem \ref{thm:compacidad}, according to which the differences $K_j=\B M_{\sigma(j)}-M_{\sigma(j)}\B$ are compact. Thus, 
$$\aligned
T_\sigma
&=M_{\sigma(1)}C_{\sigma(1)} \B \,  M_{\sigma(2)}C_{\sigma(2)} \B\, \dots  M_{\sigma(N)}C_{\sigma(N)} \B\\
&=M_{\sigma(1)}C_{\sigma(1)} M_{\sigma(2)} \B \,  C_{\sigma(2)} M_{\sigma(3)}\, \dots  \B C_{\sigma(N)} \B+K_\sigma
\endaligned$$
where all the factors containning $K_j$ are includded in $K_\sigma$. In particular, $K_\sigma$ is compact. Now, by reminding that 
$$\aligned
&\mathbf{C}\,\B=\B^\ast\,\mathbf{C},
\endaligned
$$
we have that $\B C_{\sigma(j+1)}=C_{\sigma(j+1)} B_j$ for some $B_j \in\left\{\B,\B^\ast\right\}$. Thus
$$\aligned
T_\sigma
&=M_{\sigma(1)}C_{\sigma(1)} M_{\sigma(2)} C_{\sigma(2)} B_1\,M_{\sigma(3)}\, \dots C_{\sigma(N)} B_{N-1} \B+K_\sigma
\endaligned$$
Now, one can start again. On one hand, the differences $B_j\,M_{\sigma(j+2)}-M_{\sigma(j+2)}B_j$ are again compact, because $B_j\in\{\B,\B^\ast\}$ and $M_{\sigma(j+2)}\in VMO$. Moreover, the composition $B_jC_{\sigma(j+2)}$ can be writen as $C_{\sigma(j+2)}\tilde{B}_j$, where $\tilde{B}_j$ need not be the same as $B_j$ but still $\tilde{B}_j\in\{\B,\B^\ast\}$. So, with a little abbuse of notation, and after repeating this algorythm a total of $N-1$ times, one obtains \eqref{rightorder}. The claim follows.
\end{proof}

\begin{proof}[Proof of Theorem \ref{conjug}]
The equation we want to solve can be rewritten, at least formally, in the following terms
$$(\Id-\mu\B-\nu\overline{\B})(\overline\partial f)=g,$$
so that we need to understand the $\R$-linear operator $T=\Id-\mu\B-\nu\overline\B$. By Lemma \ref{generalfredholm}, we know that $T$ is a Fredholm operator in $L^p(\omega)$, $1<p<\infty$. Now, we prove that it is also injective. Indeed, if
$$T(h)=0$$
for some $h\in L^p(\omega)$ and $\omega\in A_p$, it then follows that
$$h=\mu\B(h)+\nu\overline\B(h)$$
so that $h$ has compact support, and thus $h\in L^{1+\epsilon}(\C)$ for some $\epsilon>0$. We are then reduced to show that
$$T:L^{1+\epsilon}(\C)\to L^{1+\epsilon}(\C)\hspace{1cm}\text{is injective.}$$
Let us first see how the proof finishes. Injectivity of $T$ in $L^{1+\epsilon}(\C)$ gives us that $h=0$. Therefore, $T$ is injective also in $L^p(\omega)$. Being as well Fredholm, it is also surjective, so by the open map Theorem it has a bounded inverse $T^{-1}:L^p(\omega)\to L^p(\omega)$. % In particular the estimate
%$$\|(\Id-\mu\B -\nu\overline\B)f\|_{L^p(\omega)}\geq C_p\,\|f\|_{L^p(\omega)}$$
%holds for every $f\in L^p(\omega)$. 
As a consequence, given any $g\in L^p(\omega)$, the function
$$f={\mathcal C}T^{-1}(g)$$
is well defined, and has derivatives in $L^p(\omega)$ satisfying the estimate
$$\aligned
\|Df\|_{L^p(\omega)}
&\leq\|\partial f\|_{L^p(\omega)}+\|\overline\partial f\|_{L^p(\omega)}\\
&=\|\B T^{-1}(g)\|_{L^p(\omega)}+\|T^{-1}(g)\|_{L^p(\omega)}\\
&\leq (C+1)\,\|T^{-1}(g)\|_{L^p(\omega)}\\&\leq C\,\|g\|_{L^p(\omega)},
\endaligned$$
because $\omega\in A_p$. Moreover, we see that $f$ solves the inhomogeneous equation
$$
\overline\partial f(z)-\mu(z)\,\partial f(z)-\nu(z)\,\overline{\partial f(z)}=g(z).
$$
Finally, if there were two such solutions $f_1$, $f_2$, then their difference $F=f_1-f_2$ solves the homogeneous equation, and also $DF\in L^p(\omega)$. Thus
$$T(\overline\partial F)=0.$$
By the injectivity of $T$ we get that $\overline\partial F=0$, and from $DF\in L^p(\omega)$ we get that  $\partial F=0$, whence $F$ must be a constant.\\
We now prove the injectivity of $T$ in $L^p(\C)$, $1<p<\infty$. First, if $p\geq 2$ and $h\in L^p(\C)$ is such that $T(h)=0$, then $h$ has compact support, whence $h\in L^2(\C)$. But $\B,\overline\B$ are isometries in $L^2(\C)$, whence
$$\|h\|_2\leq k\,\|\B h \|_2=k\|f\|_2$$
and thus $h=0$, as desired. For $p<2$, we recall from Lemma \ref{adjoints} that the bijectivity of $T$ in $L^p(\C)$ is equivalent to that of $T'=\Id-\overline{\mu}\B^\ast-\nu\overline{\B^\ast}$ in the dual space $L^{p«}(\C)$. For this, note that the injectivity of $T'$ in $L^{p'}(\C)$ follows as above (since $p'\geq 2$). Note also that, by Lemma \ref{generalfredholm} we know that $T'$ is a Fredholm operator in $L^{p'}(\C)$, since $\overline\mu$ and $\nu$ are compactly supported $VMO$ functions. The claim follows. 
\end{proof}

\section{Applications}\label{BMOconn}

\noindent
We start this section by recalling that if $\mu,\nu\in L^\infty(\C)$ are compactly supported with $\||\mu|+|\nu|\|_\infty\leq k<1$ then the equation
$$\overline\partial \phi(z)-\mu(z)\,\partial\phi(z)-\nu(z)\,\overline{\partial\phi(z)}=0$$
admits a unique homeomorphic $W^{1,2}_{loc}(\C)$ solution $\phi:\C\to\C$ such that $|\phi(z)-z|\to 0$ as $|z|\to\infty$. We call it the \emph{principal} solution, and it defines a global $K$-quasiconformal map, $K=\frac{1+k}{1-k}$.\\
\\
Applications of Theorem \ref{main} are based in the following change of variables lemma, which is already proved in  \cite[Lemma 14]{AstIwaSak}. We rewrite it here for completeness.

\begin{lema}\label{conj4}
Given a compactly supported function $\mu\in L^\infty(\C)$ such that $\|\mu\|_\infty\leq k<1$, let  $\phi$ denote the principal solution to the equation
$$\overline\partial \phi(z)-\mu(z)\,\partial\phi(z)=0.$$
For a fixed weight $\omega$, let us define
$$\eta(\zeta)=\omega(\phi^{-1}(\zeta))\,J(\zeta,\phi^{-1})^{1-\frac{p}2}.$$
The following statements are equivalent:
\begin{enumerate}
\item[\emph{(a)}] For every $h\in L^p(\omega)$, the inhomogeneous equation 
\begin{equation}\label{eqn11}\overline\partial f(z)-\mu(z)\,\partial f(z)=h(z)\end{equation}
has a solution $f$ with $Df\in L^p(\omega)$ and  
\begin{equation}\label{eqn21}
\|Df\|_{L^p(\omega)}\leq C_1\, \|h\|_{L^p(\omega)}.
\end{equation}
\item[\emph{(b)}] For every $\tilde{h}\in L^p(\eta)$, the equation 
\begin{equation}\label{eqn12}\overline\partial g(\zeta)=\tilde{h}(\zeta)\end{equation}
has a solution $g$ with $Dg\in L^p(\eta)$ and
\begin{equation}\label{eqn22}
\|Dg\|_{L^p(\eta)}\leq C_2\, \|\tilde{h}\|_{L^p(\eta)}.
\end{equation}
\end{enumerate}
\end{lema}
\begin{proof}
Let us first assume that (b) holds. To get (a), we have to find a solution $f$ of \eqref{eqn11} such that $Df\in L^p(\omega)$ with the estimate \eqref{eqn21}. To this end, we make in \eqref{eqn11} the change of coordinates $g=f\circ\phi^{-1}$. We obtain for $g$ the following equation
\begin{equation}\label{eqng}\overline\partial g(\zeta)=\tilde{h}(\zeta),\end{equation}
where $\zeta=\phi(z)$ and
$$\aligned
\tilde{h}(\zeta)
%&=\frac{h(z)
%\left( \overline{\partial\phi^{-1}}+ \overline{\tilde\mu}\,\overline\partial\phi^{-1}\right)
%-\overline{h(\phi^{-1})}\tilde\nu\,\overline\partial\phi^{-1}
%}{ 1-|\tilde\mu|^2+|\tilde{\nu}|^2}\\
&= h(z)\, \frac{\partial\phi(z)}{J(z,\phi) } 
\endaligned.
$$
In order to apply the assumption (b), we must check that $\tilde{h}\in L^p(\eta)$. However,
$$\aligned
\|\tilde{h}\|_{L^p(\eta)}^p=\int |\tilde{h}(\zeta)|^p\,\eta(\zeta)\mathrm{d}\zeta
&=\int |\tilde{h}(\phi(z))|^p\,\omega(z)\,J(z,\phi)^\frac{p}2\,\mathrm{d}z\\
%&=\int\left|\frac{h(z)\left((1-|\mu(z)|^2)\frac{\partial\phi(z)}{J(z,\phi)}
%-\overline{\mu(z)}\,\nu(z)\frac{\overline{\partial\phi}(z)}{J(z,\phi)}\right)+\overline{h(z)}\left( \nu(z)^2\frac{\overline{\partial\phi}(z)}{J(z,\phi)}+ \mu(z)\nu(z)\frac{\partial\phi(z)}{J(z,\phi)}\right)}{ 1-| \mu(z)|^2+| \nu(z)|^2}\right|^2\omega(z)\,J(z,\phi)\,\mathrm{d}z\\
%&=\int\left| h(z)\left( \frac{\partial\phi(z)}{J(z,\phi)}-\overline{\mu(z)}\,\frac{\overline\partial\phi(z)}{J(z,\phi)}\right)\right|^p\frac{\omega(z)\,J(z,\phi)^\frac{p}2\mathrm{d}z}{ (1-| \mu(z)|^2)^p} \\
&=\int \left| h(z)\,   \right|^p\,\frac{\omega(z) }{(1-|\mu(z) |^2)^\frac{p}{2}}\,\mathrm{d}z\leq \frac{1}{(1-k^2)^\frac{p}2}\,\|h\|_{L^p(\omega)}^p.
\endaligned$$
Since $\tilde{h}\in L^p(\eta)$, (b) applies, and a solution $g$ to \eqref{eqng} can be found with the estimate
$$\|Dg\|_{L^p(\eta)}\leq C_2\,\|\tilde{h}\|_{L^p(\eta)}\leq \frac{C_2}{(1-k^2)^\frac{1}2}\,\|h\|_{L^p(\omega)}.$$
With such a $g$, the function $f=g\circ\phi$ is well defined, and 
$$\aligned
\int|Df(z)|^p\,\omega(z)\,\mathrm{d}z&=\int|Dg(\phi(z))\,D\phi(z)|^p\,\omega(z)\,\mathrm{d}z\\
&=\int|Dg(\zeta)\,D\phi(\phi^{-1}(\zeta))|^p\,\omega(\phi^{-1}(z))\,J(\zeta,\phi^{-1})\mathrm{d}\zeta\\
&\leq\left(\frac{1+k}{1-k}\right)^\frac{p}2\int|Dg(\zeta)|^p\,J(\phi^{-1}(\zeta),\phi)^\frac{p}{2}\,\omega(\phi^{-1}(z))\,J(\zeta,\phi^{-1})\mathrm{d}\zeta\\
&=\left(\frac{1+k}{1-k}\right)^\frac{p}2\int|Dg(\zeta)|^p\, \eta(\zeta)\,\mathrm{d}\zeta.
\endaligned$$
due to the $\frac{1+k}{1-k}$-quasiconformality of $\phi$. Moreover, $f$ satisfies the desired equation, and so 1 follows, with constant $C_1=\frac{C_2}{1-k}$.\\
To show that (a) implies (b), for a given $\tilde{h}\in L^p(\eta)$ we have to find a solution of \eqref{eqn12} satisfying the estimate \eqref{eqn22}. Since this is a $\overline\partial$-equation, this could be done by simply convolving $\tilde{h}$ with the Cauchy kernel $\frac{1}{\pi z}$. However, the desired estimate for the solution $g$ cannot be obtained in this way, because at this point the weight $\eta$ is not known to belong to $A_p$. So we will proceed in a different maner. Namely, we make the change of coordinates $f=g\circ\phi$. We obtain for $f$ the equation
$$\overline\partial f(z)-\mu(z)\,\partial f(z) =h(z),$$
where $h(z)=\tilde{h}(\phi(z))\,\overline{\partial\phi(z)}\,(1-|\mu(z)|^2)$. Moreover, 
$$
\int |h(z)|^p\,\omega(z)\,\mathrm{d}z=\int|\tilde{h}(\zeta)|^p\, (1-|\mu(\phi^{-1}(\zeta))|^2)^{p/2}\,\eta(\zeta)\,\mathrm{d}\zeta\leq \int|\tilde{h}(\zeta)|^p\,\eta(\zeta)\,\mathrm{d}\zeta.
$$ 
Therefore (a) applies, and a solution $f$ can be found with $Df\in L^p(\omega)$ and $\|Df\|_{L^p(\omega)}\leq C_1\,\|\tilde{h}\|_{L^p(\eta)}$. As before, once $f$ is found, one simply constructs $g=f\circ\phi^{-1}$. By the chain rule,
$$\aligned
\int|Dg(\zeta)|^p\eta(\zeta)\mathrm{d}\zeta
&=\int|Dg(\phi^{-1}(z))|^pJ(z,\phi^{-1})\eta(\phi^{-1}(z))\mathrm{d}z\\
&=\int|D(g\circ\phi^{-1})(z)\,(D\phi^{-1}(z))^{-1}|^pJ(z,\phi^{-1})\eta(\phi^{-1}(z))\mathrm{d}z\\
&\leq \int|Df(z)|^p\,|D\phi(\phi^{-1}(z))|^pJ(z,\phi^{-1})\eta(\phi^{-1}(z))\mathrm{d}z\\
&\leq \left(\frac{1+k}{1-k}\right)^\frac{p}{2}\int|Df(z))|^p\,J(\phi^{-1}(z),\phi)^\frac{p}2J(z,\phi^{-1})\eta(\phi^{-1}(z))\mathrm{d}z\\
&= \left(\frac{1+k}{1-k}\right)^\frac{p}{2}\int|Df(z))|^p\,\omega(z)\mathrm{d}z.
\endaligned$$
Thus,  $\|Dg\|_{L^p(\eta)}\leq C_2\,\|\tilde{h}\|_{L^p(\eta)}$ with $C_2=\left(\frac{1+k}{1-k}\right)^\frac12\,C_1$, and (b) follows.
\end{proof}

\noindent
According to the previous Lemma, a priori estimates for $\overline\partial -\mu\,\partial$ in $L^p(\omega)$ are equivalent to a priori estimates for $\overline\partial$ in $L^p(\eta)$. However, by Theorem \ref{main}, if $\omega$ is taken in $A_p$, the first statement holds, at least, when $\mu$ is compactly supported and belongs to $VMO$. We then obtain the following consequence.

\begin{coro}\label{corol}
Let $\mu\in VMO$ be compactly supported, such that $\|\mu\|_\infty<1$, and let $\phi$ be the principal solution of 
$$\overline\partial \phi(z)-\mu(z)\,\partial \phi(z)=0.$$
If $1<p<\infty$ and $\omega\in A_p$, then the weight 
$$\eta(z)=\omega(\phi^{-1}(z))\,J(z,\phi^{-1})^{1-p/2}$$ 
belongs to $A_p$. Moreover, its $A_p$ constant $[\eta]_{A_p}$ can be bounded in terms of $\mu$, $p$ and $[\omega]_{A_p}$.
\end{coro}
\begin{proof}
Under the above assumptions, by Theorem \ref{main}, we know that if $h\in L^p(\omega)$ then the equation $\overline\partial f-\mu\,\partial f=h$ can be found a solution $f$ with $Df\in L^p(\omega)$ and such that $\|Df\|_{L^p(\omega)}\leq C_0\,\|h\|_{L^p(\omega)}$, for some constant $C_0>0$ depending on $k,p$ and $[\omega]_{A_p}$. Equivalently, by Lemma \ref{conj4}, for every $\tilde{h}\in L^p(\eta)$ we can find a solution $g$ of the inhomogeneous Cauchy-Riemann equation
$$\overline\partial g=\tilde{h},$$
with $Dg\in L^p(\eta)$ and in such a way that the estimate
$$\|Dg\|_{L^p(\eta)}\leq C\,\|\tilde{h}\|_{L^p(\eta)}$$
holds for some constant $C$ depending on $C_0$, $k$ and $p$. Now, let us choose $\varphi\in{\mathcal C}^\infty_0(\C)$ and set $\tilde{h}=\overline\partial\varphi$. Then of course $g=\varphi$ and $\partial \varphi=\B(\overline\partial \varphi)$, and the above inequality says that
$$\||\partial\varphi|+|\overline\partial\varphi|\|_{L^p(\eta)}\leq C\,\|\overline\partial\varphi\|_{L^p(\eta)},$$
whence the estimate
\begin{equation}\label{finalbeur}
\|\B(\psi)\|_{L^p(\eta)}\leq (C^p-1)^\frac1p\,\|\psi\|_{L^p(\eta)}
\end{equation}
holds for any $\psi\in{\mathcal D}^\ast=\{\psi\in{\mathcal C}^\infty_c(\C); \int\psi=0\}$. It turns out that ${\mathcal D}^\ast$ is a dense subclass of $L^p(\eta)$, provided that $\eta\in L^1_{loc}$ is a positive function with infinite mass. But this is actually the case. Indeed, one has
$$
\int_{D(0,R)}\eta(\zeta)\,\mathrm{d}\zeta=\int_{\phi^{-1}(D(0,R))}\omega(z)\,J(z,\phi)^\frac{p}2\,\mathrm{d}z.
$$
Above, the integral on the right hand side certainly grows to infinite as $R\to\infty$. Otherwise, one would have that $J(\cdot,\phi)^\frac12\in L^p(\omega)$. But $\phi$ is a principal quasiconformal map, hence $J(z,\phi)=1+O(1/|z|^2)$ as $|z|\to\infty$. Thus for large enough $N>M>0$, 
$$
\int_{M<|z|<N}J(z,\phi)^\frac{p}{2}\,\omega(z)\,\mathrm{d}z\geq C\, \int_{M<|z|<N} \omega(z)\,\mathrm{d}z
$$
and the last integral above blows up as $N\to\infty$, because $\omega$ is an $A_p$ weight. \\
Therefore, the estimate \eqref{finalbeur} holds for all $\psi$ in $L^p(\eta)$. By \cite[Ch. V, Proposition 7]{Ste2}, this implies that $\eta\in A_p$, and moreover, $[\eta]_{A_p}$ depends only on the constant $(C^p-1)^\frac1p$, that is, on $k$, $p$ and $[\omega]_{A_p}$.
\end{proof}
\noindent
The above Corollary is especially interesting in two particular cases. First, for the constant weight $\omega=1$ the above result says that
$$J(\cdot,\phi^{-1})^{1-p/2}\in A_p,\hspace{2cm}1<p<\infty.$$
Without the $VMO$ assumption, this is only true for the smaller range $1+\|\mu\|_\infty<p<1+\frac{1}{\|\mu\|_\infty}$ (see e.g. \cite[Theorem 13.4.2]{AsIwMa}). Secondly, by setting $p=2$ in Corollary \ref{corol}  we get the following.

\begin{coro}\label{weightcomp}
Let $\mu\in VMO$ be compactly supported, and assume that $\|\mu\|_\infty<1$. Let $\phi$ be the principal solution of 
$$\overline\partial \phi(z)-\mu(z)\,\partial \phi(z)=0.$$
Then, for every $\omega\in A_2$ one has $\omega\circ\phi^{-1}\in A_2$.
\end{coro}

\noindent
The above result drives us to the problem of finding what homeomorphisms $\phi$ preserve the $A_p$ classes under composition with $\phi^{-1}$. Note that preserving $A_p$ forces also the preservation of the space $BMO$ of functions with bounded mean oscillation, and thus such homeomorphisms $\phi$ must be quasiconformal \cite{Rei2}. However, at level of Muckenhoupt weights, the question is deeper. As an example, simply consider the weight
$$
\omega(z)=\frac{1}{|z|^\alpha},\quad 
$$
and its composition with the inverse of a radial stretching $\phi(z)=z|z|^{K-1}$. It is clear that the values of $p$ for which $A_p$ contains $\omega$ and $\omega\circ\phi^{-1}$ are \emph{not} the same, whence preservation of $A_p$ requires something else. This question was solved by Johnson and Neugebauer \cite{JohNeu} as follows.

\begin{thm}\label{thm:gog}
Let $\phi:\mathbb C\to \mathbb C$ be $K$-quasiconformal. Then, the following statements are equivalent:
\begin{enumerate}
\item If $\omega\in A_2$ then $\omega\circ\phi^{-1}\in A_2$ quantitatively.
\item For a fixed $p\in(1,\infty)$, if $\omega\in A_p$ then $\omega\circ \phi^{-1}\in A_p$ quantitatively.
\item $J(\cdot,\phi^{-1})\in A_p$ for every $p\in(1,\infty)$.
\end{enumerate}
\end{thm}

\noindent
It follows from Corollary \ref{weightcomp} and Theorem \ref{thm:gog} that, if $\mu\in VMO$ is compactly supported, $\|\mu\|_\infty\leq k<1$ and $\phi$ is the principal solution to the $\C$-linear equation $\overline\partial\phi=\mu\,\partial\phi$, then  
$$J(\cdot,\phi^{-1})\in\bigcap_{p>1}A_p.$$
Note that if the $VMO$ assumption is removed, then we can only guarantee 
$$J(\cdot,\phi^{-1})\in\bigcap_{p>\frac{1+k}{1-k}}A_p.$$
It is not clear to the authors what is the role of $\C$-linearity in the above results concerning the regularity of the jacobian. In other words, there seems to be no reason for Theorem \ref{weightcomp} to fail if one replaces the $\C$-linear equation by the generalized one, while mantainning the ellipticity, compact support and smoothness on the coefficients. In fact, there is a deep connection between this question and the problem of determining those weights $\omega>0$ for which the estimate
$$\|Df\|_{L^2(\omega)}\leq C\,\|\overline\partial f-\mu\,\partial f-\nu\,\overline{\partial f}\|_{L^2(\omega)}$$
holds for any $f\in{\mathcal C}^\infty_0(\C)$. The following result, which is a counterpart of Lemma \ref{conj4}, explains this connection.

% That is, one could think that the same results hold for the composition operator $\omega\mapsto\omega\circ\phi^{-1}$ with the solution to \emph{any} Beltrami equation with $VMO$ coefficients, not necessarily $\C$-linear. To describe it in a different way, we are interested in finding what first order elliptic operators can be used to characterize the Muckenhoupt 

%This enchouraged us to seek for a generalized version of Lemma \ref{conj4}, which we state below.

\begin{lema}\label{conj2}
To each pair $\mu,\nu\in L^\infty(\C)$ of compactly supported functions with $\||\mu|+|\nu|\|_\infty\leq k<1$, let us associate, on one hand, the principal solution $\phi$ to the equation
$$\overline\partial \phi(z)-\mu(z)\,\partial\phi(z)-\nu(z)\,\overline{\partial\phi(z)}=0,$$
and on the other, the function $\lambda$ defined by $\lambda\circ\phi=\frac{-2i\nu}{1-|\mu|^2+|\nu|^2}$. For a fixed weight $\omega$, let us define
$$\eta(\zeta)=\omega(\phi^{-1}(\zeta))\,J(\zeta,\phi^{-1})^{1-\frac{p}2}.$$
The following statements are equivalent:
\begin{enumerate}
\item[\emph{(a)}] For every $h\in L^p(\omega)$, the equation 
$$\overline\partial f(z)-\mu(z)\,\partial f(z)-\nu(z)\,\overline{\partial f(z)}=h(z) $$ 
has a solution $f$ with $Df\in L^p(\omega)$ and $\|Df\|_{L^p(\omega)}\leq C\, \|h\|_{L^p(\omega)}$.
\item[\emph{(b)}] For every $\tilde{h}\in L^p(\eta)$, the equation 
$$\overline\partial g(\zeta)-\lambda(\zeta)\,\Im(\partial g(\zeta))=\tilde{h}(\zeta)$$ 
has a solution $g$ with $Dg\in L^p(\eta)$ and $\|Dg\|_{L^p(\eta)}\leq C\, \|\tilde{h}\|_{L^p(\eta)}$.
\end{enumerate}
\end{lema}
\noindent
Although the proof requires quite tedious calculations, it follows the scheme of Lemma \ref{conj4}, and thus we omit it. From this Lemma, we would be very interested in answering the following question. 

\begin{quest}\label{2ndquest}
Let $\omega\in L^1_{loc}(\C)$ be such that $\omega(z)>0$ almost everywhere, and let $\lambda\in L^\infty(\C)$ be a compactly supported $VMO$ function, such that $\|\lambda\|_\infty<1$. If the estimate
$$\|Df\|_{L^p(\omega)}\leq C\,\|\overline\partial f-\lambda\,\Im(\partial f)\|_{L^p(\omega)}$$
holds for every $f\in{\mathcal C}^\infty_0$, is it true that $\omega\in A_2$? 
\end{quest}
\noindent
What we actually want is to find planar, elliptic, first order differential operators, different from the $\overline\partial$, that can be used to characterize the Muckenhoupt classes $A_p$. In this direction, an affirmative answer tho Question \ref{2ndquest} would allow us to characterize $A_2$ weights as follows: given $\mu,\nu\in VMO$ uniformly elliptic and compactly supported, a positive a.e. function $\omega\in L^1_{loc}$ is an $A_2$ weight if and only if there is a constant $C>0$ such that
\begin{equation}\label{aprioriweight}
\|Df\|_{L^2(\omega)}\leq C\,\|\overline\partial f-\mu\,\partial f-\nu\,\overline{\partial f}\|_{L^2(\omega)},\hspace{1cm}\text{ for every }f\in{\mathcal C}^\infty_0(\C).
\end{equation}
Note that if $\||\mu|+|\nu|\|_\infty<\epsilon$ is small enough, \eqref{aprioriweight} says that
$$\|\partial f\|_{L^2(\omega)}^2+\|\overline\partial f \|_{L^2(\omega)}\leq C\,\|\overline\partial f\|_{L^2(\omega)}+C\,\epsilon\,\|\partial f\|_{L^2(\omega)},$$
so if $\epsilon<\frac1C$ one easily gets that
$$\| \partial f \|_{L^2(\omega)}\leq \frac{C-1}{1-C\epsilon}\|\overline\partial f \|_{L^2(\omega)}.$$
From the above estimate, weighted bounds for $\B$ easily follow, and so if $\||\mu|+|\nu|\|_\infty<\epsilon$ then such a characterization holds. Question \ref{2ndquest} has an affirmative answer.\\
\\
{{\textbf{Acknowledgements.}} Sections 1 and 2 of the present paper were obtained by the second author in \cite{Cru}, jointly with J. Mateu and J. Orobitg. The authors were partially supported by projects 2009-SGR-420 (Generalitat de Catalunya), 2010-MTM-15657 (Spanish Ministry of Science) and NF-129254 (Programa Ram\'on y Cajal). 

\bibliographystyle{plain}
\bibliography{biblio2}

\vspace{1cm}
\noindent
Albert Clop\\
albertcp@mat.uab.cat\\
Ph. +34\,93\,581\,4541\\
\\
V\'\i ctor Cruz\\
vicruz@mat.uab.cat\\
Ph. +34\,93\,581\,3739\\
\\
Departament de Matem\`atiques\\
Facultat de Ci\`encies\\
Campus de la U.A.B.
08193-Bellaterra\\
Barcelona (Catalonia)\\
Fax +34\,581\,2790

\end{document}